\documentclass[reqno]{amsart}
\usepackage{amssymb}
\usepackage[pagebackref,hypertex]{hyperref}%
\allowdisplaybreaks[4]
\theoremstyle{plain}
\newtheorem{thm}{Theorem}
\newtheorem{lem}{Lemma}

\theoremstyle{remark}
\newtheorem{rem}{Remark}
\DeclareMathOperator{\td}{d}
\numberwithin{equation}{section}
\date{Complete on 29 December 2009 in Tianjin Polytechnic University}
\date{}

\begin{document}

\title[A refinement of an inequality for the gamma function]
{A refinement of a double inequality for the gamma function}

\author[F. Qi]{Feng Qi}
\address[F. Qi]{Department of Mathematics, College of Science, Tianjin Polytechnic University, Tianjin City, 300160, China}
\email{\href{mailto: F. Qi <qifeng618@gmail.com>}{qifeng618@gmail.com}, \href{mailto: F. Qi <qifeng618@hotmail.com>}{qifeng618@hotmail.com}, \href{mailto: F. Qi <qifeng618@qq.com>}{qifeng618@qq.com}}
\urladdr{\url{http://qifeng618.wordpress.com}}

\author[B.-N. Guo]{Bai-Ni Guo}
\address[B.-N. Guo]{School of Mathematics and Informatics, Henan Polytechnic University, Jiaozuo City, Henan Province, 454010, China}
\email{\href{mailto: B.-N. Guo <bai.ni.guo@gmail.com>}{bai.ni.guo@gmail.com},
\href{mailto: B.-N. Guo <bai.ni.guo@hotmail.com>}{bai.ni.guo@hotmail.com}}

\begin{abstract}
In the paper, we present a monotonicity result of a function involving the gamma function and the logarithmic function, refine a double inequality for the gamma function, and improve some known results for bounding the gamma function.
\end{abstract}



\thanks{The first author was supported in part by the Science Foundation of Tianjin Polytechnic University}

\thanks{This paper was typeset using \AmS-\LaTeX}

\maketitle

\section{Introduction}

In \cite{Ivady-JMI-09-03-02}, the following double inequality was complicatedly procured: For $x\in(0,1)$,
\begin{equation}\label{Ivady-JMI-09-03-02-ineq}
\frac{x^2+1}{x+1}<\Gamma(x+1)<\frac{x^2+2}{x+2},
\end{equation}
where $\Gamma(x)$ stands for the classical Euler's gamma function which may be defined for $x>0$ by
\begin{equation}\label{egamma}
\Gamma(x)=\int^\infty_0t^{x-1} e^{-t}\td t.
\end{equation}
\par
The aim of this paper is to simply and concisely generalize, refine and sharpen the double inequality~\eqref{Ivady-JMI-09-03-02-ineq}.
\par
Our main results may be stated as the following theorem.

\begin{thm}\label{bounds-for-gamma-thm}
The function
\begin{equation}\label{gamma-ln-ratio}
\frac{\ln\Gamma(x+1)}{\ln(x^2+1)-\ln(x+1)}
\end{equation}
is strictly increasing on $(0,1)$, with the limits
\begin{equation}
\lim_{x\to0^+}\frac{\ln\Gamma(x+1)}{\ln(x^2+1)-\ln(x+1)}=\gamma
\end{equation}
and
\begin{equation}
\lim_{x\to1^-}\frac{\ln\Gamma(x+1)}{\ln(x^2+1)-\ln(x+1)}=2(1-\gamma).
\end{equation}
As a result, the double inequality
\begin{equation}\label{bounds-for-gamma-ineq}
\biggl(\frac{x^2+1}{x+1}\biggr)^\alpha<\Gamma(x+1) <\biggl(\frac{x^2+1}{x+1}\biggr)^\beta
\end{equation}
holds on $(0,1)$ if and only if $\alpha\ge2(1-\gamma)$ and $\beta\le\gamma$, where $\gamma=0.57\dotsm$ stands for Euler-Mascheroni's constant. Consequently, the double inequality
\begin{multline}\label{bounds-for-gamma-ineq-infty}
\biggl[\frac{(x-\lfloor x\rfloor)^2+1}{x-\lfloor x\rfloor+1}\biggr]^\alpha\prod_{i=0}^{\lfloor x\rfloor-1}(x-i) <\Gamma(x+1)\\*
< \biggl[\frac{(x-\lfloor x\rfloor)^2+1}{x-\lfloor x\rfloor+1}\biggr]^\beta\prod_{i=0}^{\lfloor x\rfloor-1}(x-i)
\end{multline}
holds for $x\in(0,\infty)\setminus\mathbb{N}$ if and only if $\alpha\ge2(1-\gamma)$ and $\beta\le\gamma$, where $\lfloor x\rfloor$ represents the largest integer less than or equal to $x$.
\end{thm}

In Section~\ref{sec-lem-ivady}, we cite three lemmas for proving in Section~\ref{sec-proof-ivady} Theorem~\ref{bounds-for-gamma-thm}. In Section~\ref{sec-rem-ivady}, we compare Theorem~\ref{bounds-for-gamma-thm} with several known results and pose some open problems and conjectures.

\section{Lemmas}\label{sec-lem-ivady}

In order to prove Theorem~\ref{bounds-for-gamma-thm}, we need the following lemma which can be found in \cite{AVV-Expos-89}, \cite[pp.~9--10, Lemma~2.9]{refine-jordan-kober.tex}, \cite[p.~71, Lemma~1]{elliptic-mean-comparison-rev2.tex} or closely-related references therein.

\begin{lem}\label{hospital-rule-ratio}
Let $f$ and $g$ be continuous on $[a,b]$ and differentiable on $(a,b)$ such
that $g'(x)\ne0$ on $(a,b)$. If $\frac{f'(x)}{g'(x)}$ is increasing $($or
decreasing$)$ on $(a,b)$, then so are the functions $\frac{f(x)-f(b)}{g(x)-g(b)}$ and
$\frac{f(x)-f(a)}{g(x)-g(a)}$ on $(a,b)$.
\end{lem}

We also need the following elementary conclusions.

\begin{lem}\label{0<>0}
For $x\in(0,1)$, we have
\begin{align*}
  x^4+4 x^3-2 x^2-4 x-3&<0,\\
(x-1)\bigl(x^2+2x-1\bigr)-(x+1)\bigl(x^2+1\bigr)\ln\frac{x^2+1}{x+1}&>0,\\
x^6+6x^5-3 x^4-16 x^3-21 x^2-6 x-1&<0,\\
x^5+5 x^4-2 x^3-8 x^2-7 x-1&<0,\\
5 x^7+34 x^6+27 x^5-62 x^4-205 x^3-198 x^2-83 x-6&<0.
\end{align*}
\end{lem}

\begin{proof}
For our own convenience, denote the functions above by $h_i(x)$ for $1\le i\le5$ on $[0,1]$ in order.
\par
By Descartes' Sign Rule, the function $h_1(x)$ has just one possible positive root. Since $h_1(1)=-4$ and $h_1(2)=29$, the function $h_1(x)$ is negative on $[0,1]$.
\par
A straightforward calculation gives
$$
\frac{\td}{\td x}\biggl[\frac{h_2(x)}{(x+1)(x^2+1)}\biggr]=\frac{(x-1)h_1(x)}{(x+1)^2(x^2+1)^2},
$$
so the function $\frac{h_2(x)}{(x+1)(x^2+1)}$ is strictly increasing on $[0,1]$. Due to $h_2(0)=1$, it is derived that $h_2(x)>0$ on $(0,1)$.
\par
Since
\begin{align*}
  h_3(1)&=-40, & h_3(3)&=1304,& h_4(1)&=-12,\\
  h_4(2)&=49, &  h_5(1)&=-488, & h_5(2)&=84,
\end{align*}
using Descartes' Sign Rule again yields the negativity of the functions $h_i(x)$ for $3\le i\le5$ on $(0,1)$. The proof of Lemma~\ref{0<>0} is complete.
\end{proof}

For our own convenience, we also recite the following double inequality for polygamma functions $\psi^{(k)}(x)$ on $(0,\infty)$.

\begin{lem}\label{qi-psi-ineq-lem}
The double inequality
\begin{equation}\label{qi-psi-ineq}
\frac{(k-1)!}{x^k}+\frac{k!}{2x^{k+1}}<(-1)^{k+1}\psi^{(k)}(x) <\frac{(k-1)!}{x^k}+\frac{k!}{x^{k+1}}
\end{equation}
holds for $x>0$ and $k\in\mathbb{N}$.
\end{lem}

For the proof of the inequality~\eqref{qi-psi-ineq}, please refer to~\cite[p.~131]{subadditive-qi.tex}, \cite[p.~223, Lemma~2.3]{property-psi.tex}, \cite[p.~107, Lemma~3]{theta-new-proof.tex-BKMS}, \cite[p.~853]{Guo-Qi-Srivasta-Unique.tex}, \cite[p.~55, Theorem~5.11]{bounds-two-gammas.tex}, \cite[p.~1625]{Extension-TJM-2003.tex}, \cite[p.~79]{AAM-Qi-09-PolyGamma.tex}, \cite[p.~2155, Lemma~3]{subadditive-qi-guo-jcam.tex} and closely-related references therein.

\section{Proof of Theorem~\ref{bounds-for-gamma-thm}}\label{sec-proof-ivady}

Now we are in a position to prove our main results in Theorem~\ref{bounds-for-gamma-thm}.
\par
It is easy to see that
\begin{equation}\label{ratio-gamma-ln}
\begin{gathered}
\frac{\ln\Gamma(x+1)}{\ln(x^2+1)-\ln(x+1)} =\frac{\frac1{x-1}\ln\Gamma(x+1)}{\frac1{x-1}\ln\frac{x^2+1}{x+1}} =\frac{\ln\sqrt[x-1]{\Gamma(x+1)}\,}{\ln\sqrt[x-1]{\frac{x^2+1}{x+1}}\,} \\
=\frac{\ln\sqrt[x-1]{\Gamma(x+1)}\,-\ln\sqrt[0-1]{\Gamma(0+1)}\,} {\ln\sqrt[x-1]{\frac{x^2+1}{x+1}}\,-\ln\sqrt[0-1]{\frac{0^2+1}{0+1}}\,}
=\frac{f(x)-f(0)}{g(x)-g(0)},
\end{gathered}
\end{equation}
where
\begin{gather*}
f(x)=\ln\sqrt[x-1]{\Gamma(x+1)}\,\quad \text{and}\quad g(x)=\ln\sqrt[x-1]{\frac{x^2+1}{x+1}}\,
\end{gather*}
on $[0,1]$. Easy computation and simplification yield
\begin{equation*}
\frac{f'(x)}{g'(x)}=\frac{(x+1)\bigl(x^2+1\bigr) [(x-1)\psi(x+1)-\ln\Gamma(x+1)]}{(x-1)(x^2+2x-1)-(x+1)(x^2+1)\ln\frac{x^2+1}{x+1}}
\end{equation*}
and
\begin{equation*}
\frac{\td}{\td x}\biggl[\frac{f'(x)}{g'(x)}\biggr]=\frac{(1-x)\bigl(x^4+4 x^3-2 x^2-4 x-3\bigr)q(x)} {\bigl[(x-1)\bigl(x^2+2x-1\bigr)-(x+1)\bigl(x^2+1\bigr)\ln\frac{x^2+1}{x+1}\bigr]^2},
\end{equation*}
where
\begin{multline*}
  q(x)=\ln\Gamma(x+1)-(x-1)\psi(x+1)
  -\frac{(x+1)\bigl(x^2+1\bigr)}{x^4+4 x^3-2 x^2-4 x-3}\\*
  \times\biggl[(x-1)\bigl(x^2+2 x-1\bigr) -(x+1)\bigl(x^2+1\bigr) \ln\frac{x^2+1}{x+1}\biggr] \psi'(x+1).
\end{multline*}
Further computation and simplification give
\begin{equation*}
q'(x)=\frac{(1-x)\bigl(x^2+2x-1\bigr)+(x+1)\bigl(x^2+1\bigr) \ln\frac{x^2+1}{x+1}}{(x^4+4 x^3-2 x^2-4 x-3)^2}q_1(x),
\end{equation*}
where
\begin{multline*}
q_1(x)=2 \bigl(x^6+6x^5-3 x^4-16 x^3-21 x^2-6 x-1\bigr) \psi'(x+1)\\*
    +(x+1)\bigl(x^2+1\bigr)\bigl(x^4+4 x^3-2 x^2-4 x-3\bigr)\psi''(x+1)
\end{multline*}
and satisfies
\begin{multline*}
  q_1'(x)=12 \bigl(x^5+5 x^4-2 x^3-8 x^2-7 x-1\bigr)\psi'(x+1)\\*
   +\bigl(x^4+4 x^3-2 x^2-4 x-3\bigr) \bigl[3\bigl(3x^2+2x+1\bigr)\psi''(x+1)\\*
  +(x+1)\bigl(x^2+1\bigr) \psi'''(x+1)\bigr].
\end{multline*}
By virtue of Lemmas~\ref{0<>0} and~\ref{qi-psi-ineq-lem}, we obtain
\begin{align*}
  q_1'(x)&<\bigl(x^4+4 x^3-2 x^2-4 x-3\bigr) \biggl\{(x+1)\bigl(x^2+1\bigr)\biggl[\frac{2}{(x+1)^3}+\frac{3}{(x+1)^4}\biggr]\\
   &\quad-3\bigl(3x^2+2x+1\bigr)\biggl[\frac{1}{(x+1)^2}+\frac{2}{(x+1)^3}\biggr]\biggr\}\\
   &\quad+12\bigl(x^5+5 x^4-2 x^3-8 x^2-7 x-1\bigr)\biggl[\frac{1}{x+1}+\frac{1}{2 (x+1)^2}\biggr]\\
   &=\frac{5 x^7+34 x^6+27 x^5-62 x^4-205 x^3-198 x^2-83 x-6}{(x+1)^3}\\
   &<0
\end{align*}
on $[0,1]$. So the function $q_1(x)$ is strictly decreasing on $[0,1]$. Since
$$
q_1(0)=-2\psi'(1)-3\psi''(1)=3.922\dotsm
$$
and
$$
q_1(1)=80\biggl(1-\frac{\pi ^2}{6}\biggr)-16 \psi''(2)=-45.128\dotsm,
$$
the function $q_1(x)$ has a unique zero on $(0,1)$, and so is the function $q'(x)$. As a result, the function $q(x)$ has a unique minimum on $(0,1)$. Because of
$$
q(0)=\frac{1}{3}\biggl(\frac{\pi^2}{6}-3\gamma\biggr)=-0.028\dotsm
$$
and $q(1)=0$, we obtain that $q(x)<0$ on $(0,1)$. Combining this with Lemma~\ref{0<>0} leads to
$$
\frac{\td}{\td x}\biggl[\frac{f'(x)}{g'(x)}\biggr]>0
$$
on $(0,1)$, which means that the function $\frac{f'(x)}{g'(x)}$ is strictly increasing on $(0,1)$. Furthermore, from Lemma~\ref{hospital-rule-ratio} and the equation \eqref{ratio-gamma-ln}, it follows that the function~\eqref{gamma-ln-ratio} is strictly increasing on $(0,1)$.
\par
By L'Hospital's rule, we have
\begin{align*}
\lim_{x\to0^+}\frac{\ln\Gamma(x+1)}{\ln(x^2+1)-\ln(x+1)} &=\lim_{x\to0^+}\frac{(x+1)\bigl(x^2+1\bigr)\psi(x+1)}{x^2+2 x-1}\\
&=-\psi(1)\\
&=\gamma
\end{align*}
and
\begin{align*}
\lim_{x\to1^-}\frac{\ln\Gamma(x+1)}{\ln(x^2+1)-\ln(x+1)} &=\lim_{x\to1^-}\frac{(x+1)\bigl(x^2+1\bigr)\psi(x+1)}{x^2+2 x-1}\\
&=2\psi(2)\\
&=2(1-\gamma).
\end{align*}
Hence, the double inequality~\eqref{bounds-for-gamma-ineq} and its sharpness follow.
\par
The double inequality~\eqref{bounds-for-gamma-ineq-infty} may be deduced from \eqref{bounds-for-gamma-ineq} and the recurrent formula $\Gamma(x+1)=x\Gamma(x)$ for $x>0$. The proof of Theorem~\ref{bounds-for-gamma-thm} is complete.

\section{Remarks}\label{sec-rem-ivady}

In this section, we compare Theorem~\ref{bounds-for-gamma-thm} with some known results and pose several open problems and conjectures.

\begin{rem}
It is clear that the double inequality~\eqref{bounds-for-gamma-ineq} refines the double inequality~\eqref{Ivady-JMI-09-03-02-ineq}.
Moreover, the inequality~\eqref{bounds-for-gamma-ineq} may be rearranged as
\begin{equation}\label{bounds-for-gamma-ineq-rearrang}
\frac1x\biggl(\frac{x^2+1}{x+1}\biggr)^{2(1-\gamma)}<\Gamma(x) <\frac1x\biggl(\frac{x^2+1}{x+1}\biggr)^\gamma,\quad x\in(0,1).
\end{equation}
\end{rem}

\begin{rem}\label{ivady-rem-2}
In \cite[p.~145, Theorem~2]{alzer-proc-ams-1999}, it was obtained that if $x\in(0,1)$, then
\begin{equation}\label{alzer-(4.1)}
  x^{\alpha(x-1)-\gamma}<\Gamma(x)<x^{\beta(x-1)-\gamma}
\end{equation}
with the best possible constants
\begin{equation}
  \alpha=1-\gamma=0.42278\dotsm\quad\text{and}\quad \beta=\frac12\biggl(\frac{\pi^2}6-\gamma\biggr)=0.53385\dotsm,
\end{equation}
and if $x\in(1,\infty)$, then \eqref{alzer-(4.1)} holds with the best possible constants
\begin{equation}
  \alpha=\frac12\biggl(\frac{\pi^2}6-\gamma\biggr)\quad \text{and}\quad \beta=1.
\end{equation}
\par
In \cite[p.~780, Corollary]{Alzer-Batir-07-AML}, the following conclusion was established: Let $\alpha$ and $\beta$ be nonnegative real numbers. For $x>0$, we have
\begin{equation}\label{alzer-batir-ineq-aml}
\sqrt{2\pi}\,x^x\exp\biggl[-x-\frac12\psi(x+\alpha)\biggr]<\Gamma(x) <\sqrt{2\pi}\,x^x\exp\biggl[-x-\frac12\psi(x+\beta)\biggr]
\end{equation}
with the best possible constants $\alpha=\frac13$ and $\beta=0$.
\par
In \cite[p.~3, Theorem~5]{note-on-li-chen.tex}, among other things, it was demonstrated that
for $x\in(0,1]$ we have
\begin{equation}\label{qi-guo-zhang-ineq-1}
\frac{x^{x[1-\ln x+\psi(x)]}}{e^x}<{\Gamma(x)}\le \frac{x^{x[1-\ln x+\psi(x)]}}{e^{x-1}}.
\end{equation}
\par
By the well-known software M\textsc{athematica} Version 7.0.0, we can show that
\begin{enumerate}
\item
the double inequalities~\eqref{bounds-for-gamma-ineq-rearrang} and \eqref{alzer-(4.1)} are not included each other on $(0,1)$,
\item
when $x>0$ is smaller, the double inequalities~\eqref{bounds-for-gamma-ineq-rearrang} is better than \eqref{alzer-(4.1)},
\item
the double inequality~\eqref{bounds-for-gamma-ineq-rearrang} improves~\eqref{alzer-batir-ineq-aml} on $(0,1)$,
\item
the left-hand side inequality in~\eqref{bounds-for-gamma-ineq-rearrang} refines the corresponding one in~\eqref{qi-guo-zhang-ineq-1},
\item
the right-hand side inequalities in~\eqref{bounds-for-gamma-ineq-rearrang} and \eqref{qi-guo-zhang-ineq-1} are not contained each other,
\item
when $x>0$ is smaller, the right-hand side inequality in~\eqref{bounds-for-gamma-ineq-rearrang} is better than the corresponding one in~\eqref{qi-guo-zhang-ineq-1}.
\end{enumerate}
\end{rem}

\begin{rem}\label{ivady-rem-3}
In \cite[Corollary~1.2, Theorem~1.4 and Theorem~1.5]{Batir-Arch-Math-08}, the following sharp inequalities for bounding the gamma function were obtained: For $x>0$, we have
\begin{gather}\label{Batir-Arch-Math-ineq-1.4}
\sqrt{2}\,\biggl(x+\frac12\biggr)^{x+1/2}e^{-x}\le\Gamma(x+1) \le e^{\gamma/e^{\gamma}}\biggl(x+\frac1{e^{\gamma}}\biggr)^{x+1/{e^{\gamma}}}e^{-x},\\
\sqrt{2e}\,\biggl(\frac{x+1/2}e\biggr)^{x+1/2}\le\Gamma(x+1) <\sqrt{2\pi}\,\biggl(\frac{x+1/2}e\biggr)^{x+1/2}\label{Batir-Arch-Math-ineq-1.5}
\end{gather}
and
\begin{multline}\label{Batir-Arch-Math-ineq-1.2}
\sqrt{2x+1}\,x^x\exp\biggl\{-\biggl[x+\frac1{6(x+3/8)}-\frac49\biggr]\biggr\}<\Gamma(x+1)\\*
 <\sqrt{\pi(2x+1)}\,x^x\exp\biggl\{-\biggl[x+\frac1{6(x+3/8)}\biggr]\biggr\}.
\end{multline}
\par
By the software M\textsc{athematica} Version 7.0.0, we can reveal that
\begin{enumerate}
\item
the double inequalities~\eqref{bounds-for-gamma-ineq} and~\eqref{Batir-Arch-Math-ineq-1.4} do not include each other on $(0,1)$,
\item
the right-hand side inequality in~\eqref{bounds-for-gamma-ineq} is better than the one in~\eqref{Batir-Arch-Math-ineq-1.5} on $(0,1)$,
\item
the left-hand side inequalities in~\eqref{bounds-for-gamma-ineq} and \eqref{Batir-Arch-Math-ineq-1.5} are not included each other on $(0,1)$,
\item
the lower bound in \eqref{bounds-for-gamma-ineq} improves the corresponding one in~\eqref{Batir-Arch-Math-ineq-1.2}, but the right-hand side inequalities in~\eqref{bounds-for-gamma-ineq} and \eqref{bounds-for-gamma-ineq} do not contain each other on $(0,1)$.
\end{enumerate}
\end{rem}

\begin{rem}
It is clear that when $x\in\mathbb{N}$ the inequality~\eqref{bounds-for-gamma-ineq-infty} becomes equality. This shows us that for $x>1$ the double inequality~\eqref{bounds-for-gamma-ineq-infty} is better than those double inequalities listed in the above Remarks~\ref{ivady-rem-2} and~\ref{ivady-rem-3}.
\end{rem}

\begin{rem}
In \cite[Theorem~1]{unit-ball.tex}, among other things, it was proved that the function
\begin{equation}\label{F(x)-dfn-Alzer}
F(x)=\frac{\ln\Gamma(x+1)}{x\ln(2x)}
\end{equation}
is both strictly increasing and strictly concave on $\bigl(\frac12,\infty\bigr)$. By L'Hospital Rule and the double inequality~\eqref{qi-psi-ineq} for $k=1$, we obtain
\begin{align*}
\lim_{x\to\infty}F(x)&=\lim_{x\to\infty}\frac{\psi(x+1)}{1+\ln(2x)} =\lim_{x\to\infty}\bigl[x\psi'(x+1)\bigr]=1,
\end{align*}
so it follows that $\Gamma(x+1)<(2x)^x$ on $\bigl(\frac12,\infty\bigr)$, which is not better than the right-hand side inequality in~\eqref{bounds-for-gamma-ineq} on $\bigl(\frac12,1\bigr)$.
\end{rem}

\begin{rem}
By similar argument to the proof of Theorem~\ref{bounds-for-gamma-thm}, we may prove that the function
\begin{equation}\label{gamma-ln-ratio-6}
  \frac{\ln\Gamma(x+1)}{\ln(x^2+6)-\ln(x+6)}
\end{equation}
is strictly decreasing on $(0,1)$. Consequently,
\begin{equation}
\biggl(\frac{x^2+6}{x+6}\biggr)^{6\gamma}<\Gamma(x+1) <\biggl(\frac{x^2+6}{x+6}\biggr)^{7(1-\gamma)}, \quad x\in(0,1).
\end{equation}
\par
Motivating by monotonic properties of the functions \eqref{gamma-ln-ratio} and \eqref{gamma-ln-ratio-6}, we pose the following open problem: What is the largest number $\lambda>1$ (or the smallest number $\lambda<6$ respectively) for the function
\begin{equation}\label{gamma-ln-ratio-lambda}
  \frac{\ln\Gamma(x+1)}{\ln(x^2+\lambda)-\ln(x+\lambda)}
\end{equation}
to be strictly increasing (or decreasing respectively) on $(0,1)$?
\end{rem}

\begin{rem}
Finally, we pose the following conjectures.
\begin{enumerate}
\item
The function \eqref{gamma-ln-ratio} is strictly increasing not only on $(0,1)$ but also on $(0,\infty)$.
\item
For $\tau>0$, the function
\begin{equation}
  \begin{cases}
    \dfrac{\ln\Gamma(x)}{\ln(x^2+\tau)-\ln(x+\tau)},& x\ne1\\
    -(1+\tau)\gamma, & x=1
  \end{cases}
\end{equation}
is strictly increasing with respect to $x\in(0,\infty)$.
\item
Recall from~\cite[Chapter~XIII]{mpf-1993}, \cite[Chapter~1]{Schilling-Song-Vondracek-2010} or~\cite[Chapter~IV]{widder} that a function $f$ is completely monotonic on an interval $I$ if $f$ has derivatives of all orders on $I$ and
\begin{equation}\label{CM-dfn}
0\le(-1)^{n}f^{(n)}(x)<\infty
\end{equation}
for $x\in I$ and $n\ge0$. We conjecture that the function
\begin{equation}\label{ln-x-frac-eq}
  h(x)=
  \begin{cases}
  \dfrac{\ln x}{\ln(1+x^2)-\ln(1+x)},&x\ne1\\
  2,&x=1
\end{cases}
\end{equation}
is completely monotonic on $(0,\infty)$.
\end{enumerate}
\end{rem}

\begin{rem}
For the history, backgrounds, origins, developments of bounding the gamma function, please refer to the expository and survey article~\cite{bounds-two-gammas.tex} and plenty of references therein.
\end{rem}

\begin{rem}
This paper is a revised version of the preprint~\cite{arxiv.org/abs/1001.1495}.
\end{rem}

\subsection*{Acknowledgements}
The authors appreciate anonymous referees for their helpful and valuable comments on this paper.

\end{document}